\documentclass[12pt]{amsart}
\usepackage{amsmath,amssymb,amsbsy,amsfonts,latexsym,amsopn,amstext,cite,
                                               amsxtra,euscript,amscd,bm}
\usepackage{url}

\usepackage{mathrsfs}

\usepackage{color}
\usepackage[colorlinks,linkcolor=blue,anchorcolor=blue,citecolor=blue,backref=page]{hyperref}
\usepackage{color}
\usepackage{graphics,epsfig}
\usepackage{graphicx}
\usepackage{float}
\usepackage{epstopdf}
\hypersetup{breaklinks=true}

\usepackage[np]{numprint}
\npdecimalsign{\ensuremath{.}}

\def\Th1{\varTheta}

\usepackage{bibentry}

\usepackage[english]{babel}
\usepackage{mathtools}
\usepackage{todonotes}
\usepackage{url}
\usepackage[colorlinks,linkcolor=blue,anchorcolor=blue,citecolor=blue,backref=page]{hyperref}

\usepackage[norefs,nocites]{refcheck}

\begin{document}

\newtheorem{theorem}{Theorem}
\newtheorem{lemma}[theorem]{Lemma}
\newtheorem{claim}[theorem]{Claim}
\newtheorem{cor}[theorem]{Corollary}
\newtheorem{conj}[theorem]{Conjecture}
\newtheorem{prop}[theorem]{Proposition}
\newtheorem{definition}[theorem]{Definition}
\newtheorem{question}[theorem]{Question}
\newtheorem{example}[theorem]{Example}
\newcommand{\hh}{{{\mathrm h}}}
\newtheorem{remark}[theorem]{Remark}

\numberwithin{equation}{section}
\numberwithin{theorem}{section}
\numberwithin{table}{section}
\numberwithin{figure}{section}

\def\sssum{\mathop{\sum\!\sum\!\sum}}
\def\ssum{\mathop{\sum\ldots \sum}}
\def\iint{\mathop{\int\ldots \int}}

\newcommand{\diam}{\operatorname{diam}}

\def\squareforqed{\hbox{\rlap{$\sqcap$}$\sqcup$}}
\def\qed{\ifmmode\squareforqed\else{\unskip\nobreak\hfil
\penalty50\hskip1em \nobreak\hfil\squareforqed
\parfillskip=0pt\finalhyphendemerits=0\endgraf}\fi}

\newfont{\teneufm}{eufm10}
\newfont{\seveneufm}{eufm7}
\newfont{\fiveeufm}{eufm5}
%
%
\newfam\eufmfam
     \textfont\eufmfam=\teneufm
\scriptfont\eufmfam=\seveneufm
     \scriptscriptfont\eufmfam=\fiveeufm
%
%
\def\frak#1{{\fam\eufmfam\relax#1}}

\newcommand{\bflambda}{{\boldsymbol{\lambda}}}
\newcommand{\bfmu}{{\boldsymbol{\mu}}}
\newcommand{\bfxi}{{\boldsymbol{\eta}}}
\newcommand{\bfrho}{{\boldsymbol{\rho}}}

\def\eps{\varepsilon}

\def\fK{\mathfrak K}
\def\fT{\mathfrak{T}}
\def\fL{\mathfrak L}
\def\fR{\mathfrak R}

\def\fA{{\mathfrak A}}
\def\fB{{\mathfrak B}}
\def\fC{{\mathfrak C}}
\def\fM{{\mathfrak M}}
\def\fS{{\mathfrak  S}}
\def\fU{{\mathfrak U}}
\def\fW{{\mathfrak W}}

\def\T {\mathsf {T}}
\def\Tor{\mathsf{T}_d}
\def\Tore{\widetilde{\mathrm{T}}_{d} }

\def\sM {\mathsf {M}}

\def\ss{\mathsf {s}}

\def\Kmnd{\cK_d(m,n)}
\def\Kmnp{\cK_p(m,n)}
\def\Kmnq{\cK_q(m,n)}

\def \balpha{\bm{\alpha}}
\def \bbeta{\bm{\beta}}
\def \bgamma{\bm{\gamma}}
\def \bdelta{\bm{\delta}}
\def \bzeta{\bm{\zeta}}
\def \blambda{\bm{\lambda}}
\def \bchi{\bm{\chi}}
\def \bphi{\bm{\varphi}}
\def \bpsi{\bm{\psi}}
\def \bnu{\bm{\nu}}
\def \bomega{\bm{\omega}}

\def \bell{\bm{\ell}}

\def\eqref#1{(\ref{#1})}

\def\vec#1{\mathbf{#1}}

\newcommand{\abs}[1]{\left| #1 \right|}

\def\Zq{\mathbb{Z}_q}
\def\Zqx{\mathbb{Z}_q^*}
\def\Zd{\mathbb{Z}_d}
\def\Zdx{\mathbb{Z}_d^*}
\def\Zf{\mathbb{Z}_f}
\def\Zfx{\mathbb{Z}_f^*}
\def\Zp{\mathbb{Z}_p}
\def\Zpx{\mathbb{Z}_p^*}
\def\cM{\mathcal M}
\def\cE{\mathcal E}
\def\cH{\mathcal H}

\def\le{\leqslant}

\def\ge{\geqslant}

\def\sfB{\mathsf {B}}
\def\sfC{\mathsf {C}}
\def\L{\mathsf {L}}
\def\FF{\mathsf {F}}

\def\sE {\mathscr{E}}
\def\sS {\mathscr{S}}

\def\cA{{\mathcal A}}
\def\cB{{\mathcal B}}
\def\cC{{\mathcal C}}
\def\cD{{\mathcal D}}
\def\cE{{\mathcal E}}
\def\cF{{\mathcal F}}
\def\cG{{\mathcal G}}
\def\cH{{\mathcal H}}
\def\cI{{\mathcal I}}
\def\cJ{{\mathcal J}}
\def\cK{{\mathcal K}}
\def\cL{{\mathcal L}}
\def\cM{{\mathcal M}}
\def\cN{{\mathcal N}}
\def\cO{{\mathcal O}}
\def\cP{{\mathcal P}}
\def\cQ{{\mathcal Q}}
\def\cR{{\mathcal R}}
\def\cS{{\mathcal S}}
\def\cT{{\mathcal T}}
\def\cU{{\mathcal U}}
\def\cV{{\mathcal V}}
\def\cW{{\mathcal W}}
\def\cX{{\mathcal X}}
\def\cY{{\mathcal Y}}
\def\cZ{{\mathcal Z}}
\newcommand{\rmod}[1]{\: \mbox{mod} \: #1}

\def\cg{{\mathcal g}}

\def\vy{\mathbf y}
\def\vr{\mathbf r}
\def\vx{\mathbf x}
\def\va{\mathbf a}
\def\vb{\mathbf b}
\def\vc{\mathbf c}
\def\ve{\mathbf e}
\def\vh{\mathbf h}
\def\vk{\mathbf k}
\def\vm{\mathbf m}
\def\vz{\mathbf z}
\def\vu{\mathbf u}
\def\vv{\mathbf v}
\def\vi{\mathbf i}
\def\vj{\mathbf j}

\def\e{{\mathbf{\,e}}}
\def\ep{{\mathbf{\,e}}_p}
\def\eq{{\mathbf{\,e}}_q}

\def\Tr{{\mathrm{Tr}}}
\def\Nm{{\mathrm{Nm}}}

 \def\SS{{\mathbf{S}}}

\def\lcm{{\mathrm{lcm}}}

 \def\0{{\mathbf{0}}}

\def\({\left(}
\def\){\right)}
\def\l|{\left|}
\def\r|{\right|}
\def\fl#1{\left\lfloor#1\right\rfloor}
\def\rf#1{\left\lceil#1\right\rceil}
\def\sumstar#1{\mathop{\sum\vphantom|^{\!\!*}\,}_{#1}}

\def\mand{\qquad \mbox{and} \qquad}

\def\tblue#1{\begin{color}{blue}{{#1}}\end{color}}




\hyphenation{re-pub-lished}

\mathsurround=1pt

\def\bfdefault{b}

\def \F{{\mathbb F}}
\def \K{{\mathbb K}}
\def \N{{\mathbb N}}
\def \Z{{\mathbb Z}}
\def \P{{\mathbb P}}
\def \Q{{\mathbb Q}}
\def \R{{\mathbb R}}
\def \C{{\mathbb C}}
\def\Fp{\F_p}
\def \fp{\Fp^*}

 \def \xbar{\overline x}

\title[Self-similar sets]{Self-similar sets with super-exponential close cylinders}

 \author[C. Chen] {Changhao Chen}

\address{Department of Mathematics, The Chinese University of Hong Kong, Shatin, Hong Kong}
\email{changhao.chenm@gmail.com}

\begin{abstract}  S. Baker (2019), B. B\'ar\'any and A. K\"{a}enm\"{a}ki (2019) independently showed that there exist iterated function systems without exact overlaps and there are super-exponentially close  cylinders at all small levels. We adapt the method of S. Baker and obtain further examples of this type.  We prove that for any algebraic number $\beta\ge 2$ there exist real numbers $s, t$  such that the iterated function system 
$$
\left \{\frac{x}{\beta}, \frac{x+1}{\beta}, \frac{x+s}{\beta}, \frac{x+t}{\beta}\right \}
$$
satisfies the above property.  
\end{abstract}


\maketitle


\section{Introduction} 

An iterated function system is a family  $\Phi=\{\varphi_i\}_{i\in \Lambda}$  of finite contraction maps  on $\R^{d}$.  Hutchinson~\cite{Hutchinson}  proved that  there exists a unique non-empty compact set $X\subset \R^{d}$ such that 
$$
X=\bigcup_{i\in \Lambda} \varphi_i(X).
$$
We call $X$ the self-similar set of $\Phi$ if all $\varphi_i$ are similarity maps. One of the fundamental problems in  fractal geometry  is to determine  the Hausdorff dimension of self-similarity sets, see~\cite{Falconer} for the definition and basic properties of Hausdorff dimension.  In this paper we use $\dim E$ to denote the  Hausdorff  dimension of the set $E$.

There are various dimensions for measuring  a given set, and  their values may be different in general, see e.g.~\cite{Falconer, Mattila2015}. However,  Falconer~\cite{Falconer89} proved that self-similar sets always  have equal Hausdorff and box dimensions.  

For any self-similar set $X$, there is a trivial upper bound for $\dim X$. More precisely, let $\Phi=\{r_i x+t_i\}_{i\in \Lambda}$  with $0<|r_i|<1, t_i\in \R^d,  i\in \Lambda$ and $\Lambda$ has finite elements. Let $\dim_S \Phi$ be the unique solution $s$ such that  $\sum_{i\in \Lambda} |r_i|^{s}=1$, then we have,  see \cite[Chapter 9]{Falconer},
\begin{equation}
\label{eq:osc}
\dim X\le \min\{d, \dim_S \Phi\}.
\end{equation}
Hutchinson \cite{Hutchinson} proved that if $\Phi$ satisfies the open set condition then the equality holds in~\eqref{eq:osc}. We remark that there are many other separation conditions which are often added on $\Phi$, see e.g. \cite{LN, Barral-Feng, NW} for more details.

We say  that there is a dimension drop for $\Phi$ if the strict inequality  holds in~\eqref{eq:osc}. 
It is not hard to show that the dimension drop occurs when  $\Phi$ has sufficiently many exact overlaps, see  \cite{Hochman2018}. 

For $\vi=i_1\ldots i_n\in \Lambda^n$ denote $\varphi_{\vi}=\varphi_{i_1} \circ \ldots \circ \varphi_{i_n}$.  We say that $\Phi$ has exact overlaps if there exist $\vi, \vj\in \Lambda^{n}$ with $\vi\neq \vj$ such that 
$
\varphi_\vi=\varphi_\vj.
$
Note that the set $\varphi_{\vi}(X)$ is often called a cylinder of  $\Phi$ or  $X$.
We remark that for $d=1$, so far the exact overlaps are the only known cause for the dimension drop. We have the following conjecture, see \cite{Simon}, and for the recent achievements, see  \cite{Rapaport, Varju}. 

\begin{conj}
\label{conj}
Let $\Phi$ be an IFS consisting of similarity maps  on $\R$. If there is a dimension drop for $\Phi$, then $\Phi$ has exact overlaps. 
\end{conj}

Hochman \cite{Hochman} proved that if there is a dimension drop for $\Phi$ then there are super-exponentially close cylinders. To be precise,  following Hochman \cite{Hochman2018} we use the following notation to quantify the amount of overlaps. Define a distance $d(\cdot,\cdot)$ between any two similarities $\varphi(x)=ax+b$ and $\varphi'(x)=a'x+b'$ by 
$$
d(\varphi, \varphi')=|b-b'|+|\log |a|-\log |a'||.
$$
For a given  IFS $\Phi=\{\varphi_i\}_{i\in \Lambda}$ and $n\in \N$ denote 
\begin{equation}
\label{eq:D_n-general}
\Delta_n=\min\{d(\varphi_{\vi}, \varphi_{\vj}): \vi, \vj \in \Lambda^{n}, \vi\neq \vj\}.
\end{equation}

Note that there are exact overlaps if and only if  $\Delta_n=0$ for some  $n$. Moreover for any IFS $\Phi$ there exists $0<c<1$ such that $\Delta_n\le c^n$ for all large enough $n$. Clearly there is an  IFS $\Phi$ such that $\Delta_n\ge c^n$ for some $c>0$ and all $n\in \N$, for instance the generating IFS $\{\frac{1}{3}x, \frac{1}{3}x+\frac{2}{3}\}$ of the Cantor ternary set. 

For an IFS $\Phi$ on $\R$ Hochman \cite{Hochman} proved that if 
$$
\dim X<\min\{1, \dim_s \Phi\},
$$
 then for any $c>0$ there exists $N$ such that for all $n\ge N$ one has 
\begin{equation}
\label{eq:super-close}
\Delta_n\le c^n.
\end{equation}
We remark that this leads to many applications in the dimension theory of self-similar sets, see \cite{Hochman} for more backgrounds and the corresponding applications. 

By the above result of Hochman and  Conjecture~\ref{conj}, it is natural to ask  whether an IFS $\Phi$ on $\R$ such that the estimate~\eqref{eq:super-close} holds (for any $c>0$ and all large enough $n$) would imply the IFS $\Phi$ has exact overlaps, see  \cite{Hochman2018}. For this direction,  Baker \cite{Baker}, B\'ar\'any and K\"{a}enm\"{a}ki \cite{BK} independently showed that there exists an IFS $\Phi$ without exact overlaps and there are super-exponentially close  cylinders at all small levels. More precisely, for any positive sequence $\{\varepsilon_n\}$ Baker \cite{Baker}, B\'ar\'any and K\"{a}enm\"{a}ki \cite{BK} showed that there exists an IFS $\Phi$ without exact overlaps such that 
\begin{equation}
\label{eq:desired-Delta_n}
\Delta_n\le \varepsilon_n, \quad \forall n\in \N.
\end{equation}

In this paper, by adapting  the method  of Baker \cite{Baker} we  obtain further examples of this type.  The IFS of  Baker  is the form 
\begin{equation} \label{eq:Baker-form}
\begin{aligned}
\Phi_{s, t}=\biggl\{&\varphi_1(x)=\frac{x}{2}, \varphi_2(x)=\frac{x+1}{2},  \varphi_3(x)=\frac{x+s}{2},\\
&\varphi_4(x)=\frac{x+t}{2}, \varphi_5(x)=\frac{x+1+s}{2}, \varphi_6(x)=\frac{x+1+t}{2} \biggr \}.
\end{aligned}
\end{equation}
Precisely, Baker \cite{Baker}  showed that for any positive sequence $\{\varepsilon_n\}$  there exists $s, t\in \R$, such that the IFS $\Phi_{s, t}$ satisfies \eqref{eq:desired-Delta_n}.  Roughly speaking,  Baker's arguments  can be divided into three steps. Firstly, for the above IFS $\Phi_{s, t}$, there is a ``nice" upper bound for $\Delta_n$ depending on the Diophantine properties of $s$ and $t$ (i.e., whether the numbers $s$ and $t$ can be approximately well by rational numbers).  Secondly,  using  the properties of continued fractions (with integer elements), Baker \cite{Baker} construct two numbers $s, t\in \R$ such that for the IFS $\Phi_{s, t}$ of \eqref{eq:Baker-form}  one has $\Delta_n\le \varepsilon_n$ for all $n\in \N$. Thirdly, show that the IFS $\Phi_{s, t}$  satisfies the desired property.

Baker \cite[Remark 2.2]{Baker} also showed that if we take the fraction $1/8$ instead of $1/2$ in \eqref{eq:Baker-form}, we can also find $s, t\in \R$ such that the IFS $\Phi_{s, t}$ satisfies \eqref{eq:desired-Delta_n}. Moreover,  \cite[Remark 2.2]{Baker}  implies that we could take any other natural numbers ($\ge 2$)  instead of $1/2$ in \eqref{eq:Baker-form} to obtain the desired IFS. Indeed this is our start point of this paper, by adapting the method of  Baker \cite{Baker} and mimicking his argument, we can take algebraic number instead of $1/2$ in  \eqref{eq:Baker-form} to obtain the desired IFS.  We also use the above three steps of Baker, however for the non-integer case $\beta$ we will use the continued fractions with non-integer elements which is the main difference from \cite{Baker}.  

Our contribution is that, except the previous  examples of Baker \cite{Baker}, B\'ar\'any and K\"{a}enm\"{a}ki \cite{BK},  we provide further IFS  without overlaps and there are super-exponentially close cylinders at all small levels.


 



 We consider the  following variant  IFS of \eqref{eq:Baker-form},
$$
\Phi_{\beta, s, t}=\left \{\frac{x}{\beta}, \frac{x+1}{\beta}, \frac{x+s}{\beta}, \frac{x+t}{\beta}\right \}.
$$
For convenience we write 
$
\Phi_{\beta, s, t}=\{\varphi_i\}_{i\in \Lambda} 
$
where $\Lambda=\{1, 2, 3, 4\}$.  Moreover,  with respect to \eqref{eq:D_n-general}, for $n\in\N$ denote
\begin{equation}
\label{eq:D_n-homo}
\Delta_n(\beta, s, t)=\min\{|\varphi_\vi(0)-\varphi_\vj(0)|: \vi, \vj\in \Lambda^n, \vi\neq \vj\}.
\end{equation}


\begin{theorem}
\label{thm:main}
Let  $\beta\ge 2$ be an algebraic number. Then  for any positive sequence $\{\varepsilon_n\}$ there exists  $\Phi_{\beta, s,t}$  without  exact  overlaps and 
$$
\Delta_n(\beta, s, t)\le \varepsilon_n, \quad \forall n\in \N.
$$
\end{theorem}

It seems that  our methods may not work for the case when  $\beta$ is a transcendental number.  The argument of B\'ar\'any and K\"{a}enm\"{a}ki \cite{BK}  may shed some new light to this situation.

We remark that (in our setting) for any algebraic $\beta\ge 2$ and the self-similar set $X$ of the IFS $\Phi_{\beta, s,t}$ of Theorem \ref{thm:main},  Rapaport \cite{Rapaport} showed that 
$$
\dim X=\min \left\{1, \frac{\log 4}{\log \beta} \right\}.
$$

We note that there are self-similar measures which are  related  to self-similar sets. The conjecture \ref{conj} can also be formulated to the self-similar measures as well, see \cite{Baker, Hochman2018, Rapaport}.

\section{Preparation}

\subsection{Notation}

Let $\beta\ge 2$ be a fixed algebraic number throughout the paper.  Denote 
\begin{equation*}
\P_\beta=\{\beta^n: n=0,1, 2, \ldots\};
\end{equation*}
\begin{equation*}
\Z_\beta=\{f(\beta): f\in \Z[x]\};
\end{equation*}
\begin{equation*}
\Q_\beta= \{f/g: f, g\in \Z_\beta, g \neq 0 \}.
\end{equation*}

For a number $f(\beta)\in \Z_\beta$ we sometimes regard $f(\beta)$ as a polynomial with the indeterminate $\beta$ when there is no confusion. Thus the degree of $f(\beta)$ is understood as the degree of $f\in \Z[x]$.

\subsection{Iterated function system $\Phi_{\beta, s, t}$}

We first introduce the following ``$\beta$-based'' sets.

\begin{equation}
\label{eq:Base-beta}
\cB_n=\left \{\sum_{j=1}^{n}\omega_j\beta^{j-1}: \omega_j\in \{0, 1\}\right \} \quad \text{and}  \quad \cB=\bigcup_{n=1}^{\infty} \cB_n.
\end{equation}

Note that by the restriction $\beta\ge 2$  any element in $\cB$ has an unique  representation. Indeed  this follows by the fact 
that for any  $k\ge 1$ we have 
$$
\beta^k>\beta^{k-1}+\ldots+1.
$$
We remark that this is 
the reason of letting $\beta\ge 2$ in Theorem~\ref{thm:main}.

In analogy of  Baker \cite[Lemma 2.1]{Baker} we have the following upper bound for $\Delta_n(\beta, s, t)$.

\begin{lemma} 
\label{lem:Delta} For IFS $\Phi_{\beta, s, t}$ and  $n\in \N$ we have 
$$
\Delta_n(\beta, s, t)\le \min\left \{ \min_{ \substack{p, q\in \cB_n \\ (p, q)\neq (0, 0)}} |qs-p|, \min_{\substack{p, q\in\cB_n \\ (p, q)\neq (0, 0)}}|qt-p|  \right \}.
$$
\end{lemma}
\begin{proof}
Observe that 
\begin{align*}
\{\varphi_\vi(0): \vi\in \Lambda^n\}= \left \{\sum_{j=1}^n c_j \beta^{-j+1}: c_j\in \{0, 1/\beta, s/\beta, t/ \beta\}\right \}.
\end{align*}

Recall that the arithmetic sums of sets $X, Y \subseteq \R$ is defined as  
$$
X+Y=\{x+y: x\in X, y\in Y\}.
$$
Moreover for $\rho\in \R$ denote 
$
\rho X=\{\rho x: x\in X\}.
$
Let $A=\{0, 1, s, t\}$ then 
\begin{equation}
\label{eq:phi0}
\{\varphi_\vi(0): \vi\in \Lambda^n\}
=\beta^{-n}(A+\beta A+\ldots +\beta^{n-1} A).
\end{equation}
Applying~\eqref{eq:Base-beta}  we derive 
\begin{equation*}
\cB_n \cup s\cB_n \cup t\cB_n \subseteq A+\beta A+\ldots +\beta^{n-1} A.
\end{equation*}
Combining with \eqref{eq:phi0} and \eqref{eq:D_n-homo} we obtain the desired bound.
\end{proof}

\begin{lemma} 
\label{lem:non}
Let $\beta\ge 2$ and  $s, t \notin \Q_\beta$. Suppose  the IFS $\Phi_{\beta, s, t}$ has exact overlaps then there exit $f(\beta), g(\beta) \in \Z_\beta \setminus\{0\}$ and $h(\beta)\in \Z_\beta$ such that 
$$
s=\frac{f(\beta)}{g(\beta)}t+\frac{h(\beta)}{g(\beta)}.
$$
\end{lemma}
\begin{proof}
Since $\Phi_{\beta, s, t}$ has exact overlaps, there is $n\in \N$ and $\vi, \vj\in \Lambda^n, \vi\neq \vj$ such that $\varphi_\vi(0)=\varphi_\vj(0)$. Applying \eqref{eq:phi0} and $\vi\neq \vj$, there exist $a_i, a_i' \in \{0, 1, s, t\}, 1\le  i\le  n $  such that $a_{i_0}\neq a_{i_0}'$ for some $1\le i_0\le n$ and 
\begin{equation*}
\sum_{i=1}^n a_i \beta^{i-1} =\sum_{i=1}^{n} a_i' \beta^{i-1}.
\end{equation*}
It follows that there are  $L_i, L_i' \in \cB_n, i=1, 2, 3 $ such that
\begin{equation}
\label{eq:L1L2L3}
L_1+ L_2s+L_3t=L_1'+L_2's+L_3't.
\end{equation}
Furthermore, since each element of $\cB_n$ has an unique representation, we conclude  that 
\begin{equation*}
(L_1, L_2, L_3) \neq (L_1', L_2', L_3').
\end{equation*} 
Combining with \eqref{eq:L1L2L3} we derive 
$
(L_2, L_3)\neq (L_2', L_3').
$
We claim that $L_2\neq L_2'$ and $L_3\neq L_3'$. Indeed assume to the contrary that   $L_2=L_2'$, then $L_3\neq L_3'$ and hence 
$$
t=\frac{L_1-L_1'}{L_3'-L_3} \in \Q_\beta,
$$
which is contradict to the assumption that $t\notin \Q_\beta$. Similar argument yields $L_3\neq L_3'$. 
 Thus  the above claim is true. Combining with \eqref{eq:L1L2L3}  we obtain the desired identity. 
\end{proof}

\subsection{Values of polynomials on algebraic numbers}

Garsia \cite[Lemma 1.51]{Garsia} first applied the  estimates of polynomials with integer  coefficients on algebraic numbers to the theory of Bernoulli convolution, since then this method and its variants have lead to many applications in fractal geometry, see e.g.\cite[Section 5]{Feng-Pisot}, \cite[Theorem 1.5]{Hochman},~\cite[Section 6]{LN}, ~\cite{Rapaport, Varju}.  

The following form is taking from~\cite[Lemma 11]{Rapaport}. Denote by $\cP(n, H)$ the collection of   integer coefficient polynomials with degree at most $n$ and  its  coefficients are bounded by $H$.

\begin{lemma}
\label{lem:zero-or-1}
For any  algebraic number $\xi$ there exists  $M>0$ depending only on $\xi$   such that for any $f\in \cP(n, H)$ if  $f(\xi)\neq 0$ then
$$
|f(\xi)|\ge M^{-n} H^{-M}.
$$
\end{lemma}

We remark that one may obtain better lower bounds for some special algebraic numbers, for instance when $\beta$ is a Pisot number or Salem number, see \cite[Lemma 1.51]{Garsia}, \cite[Section 5]{Feng-Pisot},~\cite[Section 6]{LN}.

\subsection{Continued fractions with non-integer letters}
 In this subsection we study the continued  fractions which  its letters may not be positive integers.  Our main results of this subsection are Lemma~\ref{lem:irrational} and Lemma~\ref{lem:best}.  We start from recalling some well known facts from continued fractions, see  \cite{Khinchin} for more details.  An expression of the form
\begin{equation}
\label{eq:form-finite}
\cfrac{1}{a_1+\cfrac{1}{a_2+\cfrac{1}{ a_3+\cdots \cfrac{1}{a_k}}}}
\end{equation}
is called a finite continued fraction, and  for convenience we write it as    
$$
[a_1, a_2, \ldots, a_k].
$$

In general applications of continued fractions the letters $a_1, a_2, \ldots, a_k$ are often assumed to be  positive integers. However, for our applications we always assume  $a_1, a_2, \ldots, a_k$  to be elements of  $\P_\beta$, and thus in general  the elements $a_1, a_2, \ldots, a_k$ may not be natural numbers. 

Given a  real number sequence  $\{a_n\}$ with  $a_n\ge 1$, denote $\frac{p_k}{q_k}$ the value of~\eqref{eq:form-finite}, i.e., 
$$
\frac{p_k}{q_k}=\cfrac{1}{a_1+\cfrac{1}{a_2+\cfrac{1}{ a_3+\cdots \cfrac{1}{a_k}}}}.
$$
We remark that $p_k$ and $q_k$ are the ``canonical representations'' of the finite continued fractions, see \cite[Chapter 1]{Khinchin}.

In the following we collect some useful facts for the sequence $\{\frac{p_n}{q_n}\}$ under the condition $a_n\ge 1$ for all $n\in \N$. 
\begin{itemize}
\item  For convenience denote  $p_0=0, q_0=1$.  For $k= 2, 3, \ldots$ we have 
\begin{equation} \label{eq:induction}
\begin{aligned}
&p_{k}=a_kp_{k-1}+p_{k-2}\\
&q_{k}=a_kq_{k-1}+q_{k-2}
\end{aligned}
\end{equation}

\item  The sequence  $\{\frac{p_k}{q_k}\}$ is convergent and we denote 
$$
[a_1, a_2, \ldots, a_n, \ldots]=\lim_{k\rightarrow \infty}\frac{p_k}{q_k}. 
$$
\item  Let $x:=[a_1, a_2, \ldots, a_n, \ldots]$ then
\begin{equation}
\label{eq:Dio}
\frac{1}{q_k(q_{k+1}+q_k)}<\left |x-\frac{p_k}{q_k}\right |<\frac{1}{q_kq_{k+1}}.
\end{equation}
\end{itemize}

We remark that most of the above properties follow directly from \cite[Chapter 1]{Khinchin}, but  there is an exceptional case for the left hand side of the estimate~\eqref{eq:Dio} when the numbers $a_n$ are not  integers.  We show a simper proof in the following. We suppose that $p_k/q_k$ is at the left side of $x$, similar argument woks for the right side as well. By \cite[p.6]{Khinchin} we have 
$$
\frac{p_k}{q_k}<\frac{p_{k+2}}{q_{k+2}}<x.
$$
Combining with \cite[p.6]{Khinchin} we derive 
\begin{equation}
\label{eq:ak+2}
\left | x-\frac{p_k}{q_k}\right |>\left | \frac{p_k}{q_k}-\frac{p_{k+2}}{q_{k+2}}\right |=\frac{a_{k+2}}{q_kq_{k+2}}.
\end{equation}
Applying~\eqref{eq:induction} and the condition  $a_n\ge 1$ we obtain 
$$
q_{k+2}/a_{k+2}\le q_{k+1}+q_{k}.
$$
Thus combining with \eqref{eq:ak+2} we obtain the desired  lower bound in~\eqref{eq:Dio}.



By \cite[Theorem 14]{Khinchin} if $\{a_n\}$ is a sequence of positive integers then $[a_1, a_2, \ldots]$ is an irrational number. We note that this  property is used in \cite{Baker} for showing the desired IFS  does not have  exact overlaps. For the  sequence $\{a_n\}$ with $a_n\in \P_\beta$ we have the following substitution which plays a similar  role to \cite{Baker}.

\begin{lemma}
\label{lem:irrational}
Let $\beta\ge 2$ be an algebraic number and  $\{a_n\}$ be a sequence with $ a_n\in \P_\beta, n\in \N$. Denote $p_k/q_k=[a_1, \ldots, a_k]$ and  consider $p_k, q_k$ as polynomials  with indeterminate $\beta$.  
Suppose that $p_k, q_k\in \cP(d_k, d_k)$ for some $d_k\in \N$ and 
\begin{equation}
\label{eq:ak+1}
a_{k+1}\ge M^{kd_k}d_k^M
\end{equation}
where $M$ is the same as in Lemma~\ref{lem:zero-or-1}. 
Then  $[a_1, a_2, \ldots] \notin \Q_\beta$.
\end{lemma}
\begin{proof}
Assume to the contrary that $[a_1, a_2, \ldots] \in \Q_\beta$. Then there  exit $f(\beta), g(\beta)\in \Z_\beta$ with $g(\beta)\neq 0$ such that 
$$
\frac{f(\beta)}{g(\beta)}=[a_1, a_2, \ldots].
$$ 
For each $k\in \N$ by~\eqref{eq:induction} and~\eqref{eq:Dio} we have 
\begin{equation}
\label{eq:positive}
0<\left | f(\beta)q_k-g(\beta)p_k \right |<g(\beta)/q_{k+1}<g(\beta)/a_{k+1}.
\end{equation}
Note that we can regard $f, g$ as elements of $\cP(d, d)$ for some  integer $d$. Combining with the condition
$p_k, q_k\in \cP(d_k, d_k)$, there exists $C>0$ such that 
$$
f(\beta)q_k-g(\beta)p_k\in \cP(Cd_k, Cd_k).
$$
Applying Lemma~\ref{lem:zero-or-1} and  the non-zero condition in~\eqref{eq:positive}, we obtain
$$
|f(\beta)q_k-g(\beta)p_k|\ge M^{-Cd_k}(Cd_k)^{-M},
$$
which is contradict to \eqref{eq:ak+1} and~\eqref{eq:positive}.
\end{proof}

\begin{remark}
As we claimed before for some special algebraic numbers we have better lower bounds in Lemma \ref{lem:zero-or-1}, and hence  we can obtain weak condition of $a_{k+1}$ in the above Lemma \ref{lem:irrational} for these special algebraic numbers. 

Note that   there is  sequence $\{a_n\}$ with $a_n\in \P_\beta$ such that $[a_1, a_2, \ldots]\in \Q_\beta$. For instance 
for  $a_n=\beta, n\in \N$ we have 
$$
[\beta, \beta, \ldots]=\frac{1}{\beta+[\beta, \beta, \ldots]},
$$
and hence 
$$
[\beta, \beta, \ldots]=\frac{\sqrt{\beta^2+4}-\beta}{2}.
$$
Suppose $\sqrt{\beta^2+4}\in \Z$ then $[\beta, \beta, \ldots]\in \Q_\beta$. 
\end{remark}

We remark that the best approximate property of continued fractions (with positive integer letters)  is used in the construction of \cite{Baker} for finding the desired  parameters $s$ and $t$. The best approximate property  \cite[Chapter 2]{Khinchin} claims  that for any sequence of positive integers $\{a_n\}$, denote $s=[a_1, a_2, \ldots]$ and $p_k/q_k$ be its partial quotient, and for any integers $1\le q\le q_k, p\in \Z$ one has 
$$
\left |s-\frac{p}{q}\right |\ge \left |s-\frac{p_k}{q_k}\right |.
$$ 
Clearly, there is no such best approximate property  for the general real number   sequence $\{a_n\}$.  However, if the sequence  $\{a_n\} $ is a subset of $\P_\beta$ for some algebraic number $\beta$ then we have the following  Lemma~\ref{lem:best}  which  is sufficient for our application. 

For $a_1, a_2, \ldots, a_k\in \P_\beta$  denote 
\begin{equation*}
\cC[a_1, a_2, \ldots, a_k]=\{[a_1, \ldots, a_k, a_{k+1}, \ldots]:  a_j\in \cP_\beta, j\ge k+1 \}.
\end{equation*}

\begin{lemma}
\label{lem:best}
Let $a_1, \ldots, a_k\in \P_\beta$ and $n\in \N$. Then there exits $a_{k+1}\in \P_{\beta}$  such that for any $s\in \cC[a_1, \ldots, a_{k+1}]$ and any $f, g\in \cP(n, n)$ with $g(\beta)\neq 0$ we have 
$$
\left |s-\frac{f(\beta)}{g(\beta)}\right |\ge c>0,
$$
where the constant $c$ depends on  $a_1, \ldots, a_{k+1}$ and $n$. We stress that the constant $c$ will not depend on the specific  choices of $f$ and $g$.
\end{lemma}
\begin{proof}
Denote $\frac{p_k}{q_k}=[a_1, \ldots, a_k]$.
We proceed on a case-by-case basis depending   on  $\frac{p_k}{q_k}=\frac{f(\beta)}{g(\beta)}$ or not. First suppose that 
$$
\frac{p_k}{q_k}=\frac{f(\beta)}{g(\beta)}.
$$
Then by~\eqref{eq:Dio} we obtain 
\begin{equation}
\label{eq:eq}
\left| s-\frac{f(\beta)}{g(\beta)}\right |=\left| s-\frac{p_k}{q_k}\right |>\frac{1}{q_k(q_{k+1}+q_k)},
\end{equation}
thus this gives a positive lower bound. 

We now turn to the case $\frac{p_k}{q_k}\neq \frac{f(\beta)}{g(\beta)}$. Together with~\eqref{eq:Dio} we obtain
\begin{equation}\label{eq:cc}
\begin{aligned}
\left | s-\frac{f(\beta)}{g(\beta)}\right |
&\ge\left |  \frac{p_k}{q_k}-\frac{f(\beta)}{g(\beta)}  \right | - \left |  s-\frac{p_k}{q_k}\right | \\
&\ge  \frac{|p_k g(\beta)-q_kf(\beta)|}{|q_k g(\beta)|}   - \frac{1}{q_kq_{k+1}}.
\end{aligned}
\end{equation}
By~\eqref{eq:induction} we have  $p_k, q_k \in \Z_\beta$ and we consider $p_k, q_k$ as polynomials. Therefore, there exists a large integer $L$ which depends on $a_1, \ldots, a_k, n$ such that  
$$
p_k g(\beta)-q_kf(\beta)\in \cP(L, L).
$$ 
Clearly the condition $\frac{p_k}{q_k}\neq \frac{f(\beta)}{g(\beta)}$ implies 
$p_k g(\beta)-q_kf(\beta)\neq 0$. 
Thus by Lemma~\ref{lem:zero-or-1} there exits $M$ depending only on $\beta$  such that
\begin{equation}
\label{eq:upper}
\left | p_k g(\beta)-q_kf(\beta) \right |\ge M^{-L}L^{-M}.
\end{equation}
Moreover, there exists $M_1$ depending on  $a_1, \ldots, a_k, n$ such that 
\begin{equation}
\label{eq:M1}
q_k g(\beta)\le M_1.
\end{equation}

Let  $a_{k+1}\in \P_\beta$ be a large number  such that  
$
a_{k+1}>2 M^L L^M M_1,
$
then by \eqref{eq:induction} we have 
$$
(q_kq_{k+1})^{-1}\le a_{k+1}^{-1}\le M^{-L}L^{-M}M_1^{-1}/2.
$$
Combining with~\eqref{eq:cc}, ~\eqref{eq:upper} and  ~\eqref{eq:M1}, we obtain
$$
\left | s-\frac{f(\beta)}{g(\beta)}\right| \ge M^{-L}L^{-M}M_1^{-1}/2.
$$
Together with~\eqref{eq:eq} we obtain the desired result.
\end{proof}

\section{Proofs of Theorem~\ref{thm:main}}

We mimick the construction of Baker \cite{Baker} to our setting. Without loss of generality we assume that the sequence $\{\varepsilon_n\}$ is monotone decreasing. Otherwise we may consider the sequence $\{\varepsilon_n'\}$ where 
$
\varepsilon_n'=\min_{1\le i\le n} \varepsilon_i.
$

\subsection*{Initial construction}  
Let $a_1=1$ then by the definition of continued fractions we have   $p_1, q_1=1$.  Thus $p_1, q_1\in \cB_1$ where  $\cB_n, n\in \N$ is given as in \eqref{eq:Base-beta}. Let $N_1 \ge 2$ a natural number then there exists $a_{2}\in \P_\beta$  such that for  any $s\in \cC[a_1, a_{2}]$, by \eqref{eq:Dio}, we have 
\begin{equation}
\label{eq:initial-sS}
|s-1|=|q_1s-p_1|< 1/a_{2}\le  \varepsilon_{N_1}.
\end{equation} 
By \eqref{eq:induction} there exits $n\in \N$ such that 
\begin{equation}
\label{eq:initial-beta-base}
p_2, q_2\in \cB_n.
\end{equation}
Moreover, by taking sufficiently large $a_2\in \P_\beta$ and applying  Lemma \ref{lem:zero-or-1},  there exists  a positive $c_1$ such that for any $f, g, h\in \cP(1, 1)$ with $g(\beta)\neq 0$ one has 
\begin{equation}
\label{eq:initial-sL}
\left |s-\frac{f(\beta)}{g(\beta)}-\frac{h(\beta)}{g(\beta)}\right |\ge c_1.
\end{equation}

Let  $a_2\in \P_\beta$ be sufficiently large  such that the above~\eqref{eq:initial-sS},~\eqref{eq:initial-beta-base}, and~\eqref{eq:initial-sL} hold. Denote
\begin{equation*}
N_{2}=\min\{n: p_2, q_2\in \cB_n\}.
\end{equation*}

Now we turn to the construction of $t$. Let $a_1'=1$. For $\varepsilon_{N_2}$ and $c_1$ there exists $a_2'\in \P_\beta$ such that  for  any $t\in \cC[a_1', a_2']$ we have 
\begin{equation}
\label{eq:t-1}
|t-1|=|q_1's-p_1'|< 1/a_2'\le \min\{\varepsilon_{N_2}, c_1\}.
\end{equation} 

We remark that the above conditions~\eqref{eq:initial-sL},~\eqref{eq:t-1} and their induction versions~\eqref{eq:con2},~\eqref{eq:Nk+1} below  will be used to prove  that the final IFS $\Phi_{\beta, s, t}$ does not have exact overlaps. 

Moreover, we ask that $a_2$ and $a_2'$  are sufficiently large  such that they satisfy the condition~\eqref{eq:ak+1} of Lemma~\ref{lem:irrational} respectively.

\subsection*{Iterated construction}
Let $k\ge 2$. Suppose that  there are $a_2, \ldots, a_k$ and $a_2', \ldots, a_k'$ such that  
for any 
$$
s\in \cC[a_1, \ldots, a_k] \quad \text{and} \quad t\in \cC[a_1',  \ldots, a_k']
$$ 
we have 
\begin{equation}
\label{eq:con1}
\Delta_n(\beta, s, t)\leq \varepsilon_n, \quad \forall 1\le n\le N_k.
\end{equation}
Moreover, 
\begin{equation}
\label{eq:conk}
p_k, q_k \in \cB_{N_k} \quad \text{and}\quad p_k', q_k'\in \cB_{M_k}.
\end{equation}
Furthermore  there exists a positive number $c_{k-1}$ such that for any 
$$
s\in \cC[a_1, \ldots, a_{k}] \quad \text{and} \quad t\in \cC[a_1', \ldots, a_{k-1}']
$$ 
and  any $f, g, h \in \cP(k-1, k-1)$ with $g(\beta)\neq 0$ we have 
\begin{equation}
\label{eq:con2}
\left| s-\frac{p_{k-1}'f(\beta)}{q_{k-1}'g(\beta)} -\frac{h(\beta)}{g(\beta)} \right |\ge c_{k-1}, 
\end{equation}
and 
\begin{equation*}
\left |t-\frac{p_{k-1}'}{q_{k-1}'}\right |\le \frac{c_{k-1}}{k-1}.
\end{equation*}

Note that we may add further conditions $M_k\ge 2N_k$ and $N_{k}>M_{k-1}$ to make sure that $N_k$ tends to infinity as $k$ tends to infinity. Moreover, we ask that  $a_k$ and $a_k'$  are large enough such that they satisfy the condition \eqref{eq:ak+1} of Lemma~\ref{lem:irrational} respectively.

Under the above assumption, we now begin to choose $a_{k+1}$ and $a_{k+1}'$ such that the above claims  still hold for the case  $k+1$.  

For the number $\varepsilon_{M_k}$ there exists $a_{k+1}\in \P_{\beta}$ such that for any 
$$
s\in \cC[a_1, \ldots, a_{k+1}],
$$
by~\eqref{eq:Dio} we have 
\begin{equation}
\label{eq:Mk}
|q_k s-p_k|\le 1/a_{k+1}\le \varepsilon_{M_{k}}.
\end{equation}
Combining with Lemma \ref{lem:Delta} and the condition~\eqref{eq:conk} we derive that 
\begin{equation}
\label{eq:Mkk}
\Delta_n(\beta, s, t)\le \beta^{-n} |q_ks-p_k|\le \varepsilon_{M_k}, \quad \forall N_k\le n\le M_k.
\end{equation}

Applying \eqref{eq:induction} and  the assumption $p_k, q_k\in \cB_{N_k}$, and by taking some large enough $a_{k+1}\in \P_\beta$, there exists a natural number $n>N_k$ such that  
\begin{equation}
\label{eq:it-Bn}
p_{k+1}, q_{k+1}\in \cB_{n}.
\end{equation}

Furthermore, by Lemma~\ref{lem:best} there exits   $a_{k+1}\in \P_{\beta}$ such that for any 
$$
s\in \cC[a_1, \ldots, a_{k+1}],
$$
and any $f, g, h\in \cP(k, k)$ with $g(\beta)\neq 0$ one has 
\begin{equation}\label{eq:sL}
\left |s-\frac{f(\beta)}{g(\beta)}\frac{p_k'}{q_k'}-\frac{h(\beta)}{g(\beta)}\right |\ge c_k>0,
\end{equation}
where $c_k$ depends on $\beta, k$ and  $a_2, \ldots, a_{k+1}, a_2', \ldots, a_k'$. 

We take large enough $a_{k+1}\in \P_\beta$ such that the above~\eqref{eq:Mk},~\eqref{eq:it-Bn} and~\eqref{eq:sL} still hold with respect to $a_{k+1}$. Denote 
\begin{equation*}
N_{k+1}=\min\{n\in \N: p_{k+1}, q_{k+1}\in \cB_n\}.
\end{equation*}

For $\varepsilon_{N_{k+1}}$ and $c_k$ there exists $a_{k+1}'\in \P_{\beta}$ such that for any 
$$
t\in \cC[a_1', \ldots, a_{k+1}'],
$$
by \eqref{eq:Dio} we have 
\begin{equation}
\label{eq:Nk+1}
|q_k't-p_k'|\le 1/a_{k+1}'\le \min\{\varepsilon_{N_{k+1}}, c_k/k\}.
\end{equation}
Combining with Lemma \ref{lem:Delta} and the condition~\eqref{eq:conk} we obtain 
\begin{equation}
\label{eq:Nk+1'}
\Delta_n(\beta, s, t)\le \beta^{-n}|q_k'-p_k'|\le \varepsilon_{N_{k+1}}, \quad \forall  M_k\le  n\le N_{k+1}.
\end{equation}

Moreover, by~\eqref{eq:induction} and the assumption $p_k', q_k'\in \cB_{M_k}$, by taking large enough $a_{k+1}'\in \P_\beta$ we obtain
$p_{k+1}', q_{k+1}'\in \cB_{n}$ for some $n\in \N$. Denote 
$$
M_{k+1}=\min\{n\in \N: p_{k+1}', q_{k+1}'\in \cB_n\}.
$$

Furthermore,  we ask that  $a_k$ and $a_k'$ are  sufficiently large  such that they satisfy the condition~\eqref{eq:ak+1} of Lemma~\ref{lem:irrational} respectively.

By~\eqref{eq:con1}, ~\eqref{eq:Mkk} and  ~\eqref{eq:Nk+1'} we derive that for any 
$$
s\in \cC[a_1, \ldots, a_{k+1}] \quad \text{and} \quad t\in \cC[a_1',  \ldots, a_{k+1}']
$$
we have 
$$
\Delta_n(\beta, s, t)\le \varepsilon_n, \quad \forall 1\le n\le N_{k+1}.
$$

By iterating  the above arguments, there exit two sequences $\{a_k\}$ and $\{a_k'\}$ such that  
\begin{equation*}
s=[a_1, \ldots, a_n, \ldots] \notin \Q_\beta \quad \text{and} \quad t=[a_1', \ldots, a_n', \ldots]\notin \Q_\beta,
\end{equation*}
and 
\begin{equation*}
\Delta(\beta, s, t)\le \varepsilon_n, \quad \forall n\in \N.
\end{equation*}
Moreover  for each $k\in \N$ and any $f, g, h\in \cP(k, k)$ with $g(\beta)\neq 0$ we have 
\begin{equation}
\label{eq:allk}
\left |s-\frac{f(\beta)}{g(\beta)}\frac{p_k'}{q_k'}-\frac{h(\beta)}{g(\beta)}\right |\ge c_k>0,
\end{equation}
where $c_k$ depends on $\beta, k$ and  $a_2, \ldots, a_{k+1}, a_2', \ldots, a_k'$. 
Furthermore  by \eqref{eq:Nk+1} we obtain
\begin{equation}
\label{eq:t}
\left |t-\frac{p_k'}{q_k'}\right |\leq c_k/k.
\end{equation}

\subsection*{No exact overlaps} The following arguments are due to \cite{Baker}.  For completeness we show the  arguments here.

Let $s, t$ be as in the above construction. Assume to the contrary  that $\Phi_{\beta, s, t}$ has exact overlap. Then by Lemma~\ref{lem:non} there are  $f(\beta), g(\beta)\in \Z_\beta\setminus\{0\}$ and $h(\beta)\in \Z_\beta$ such that 
\begin{equation*}
s=\frac{f(\beta)}{g(\beta)}t+\frac{h(\beta)}{g(\beta)}.
\end{equation*}
Combining with~\eqref{eq:t}, for each $k$ we have 
\begin{equation}
\label{eq:R}
\left | s-\frac{f(\beta)}{g(\beta)}\frac{p_k'}{q_k'}-\frac{h(\beta)}{g(\beta)} \right |\le \left |\frac{f(\beta)}{g(\beta)}\right | \left |t-\frac{p_k'}{q_k'}\right |\le \frac{c_k|f(\beta)|}{k|g(\beta)|}.
\end{equation}
Other the other hand, we could consider $f(\beta), g(\beta), h(\beta)$ as polynomials and they belongs to $\cP(k, k)$ for all large enough $k$. Therefore,  by~\eqref{eq:allk} we obtain 
$$
\left | s-\frac{f(\beta)}{g(\beta)}\frac{p_k'}{q_k'}-\frac{h(\beta)}{g(\beta)} \right |\ge c_k,
$$
which is contradict to~\eqref{eq:R} when $k$ is large enough.

\section*{Acknowledgement}

It is my pleasure to thank the members of the online reading group on  fractal geometry in the  department of mathematics. In particular, I am grateful to De-Jun Feng for the very careful reading of the manuscript and many valuable suggestions. This work  was supported  by HKRGC GRF Grants  CUHK14301218 and CUHK14304119.


\end{document}